\documentclass[12 pt]{amsart}
\usepackage{amssymb,pdfsync}
\usepackage{graphicx}
\usepackage[mathscr]{eucal}
\usepackage{dsfont}
\usepackage[usenames]{color}
\definecolor{Purple}{rgb}{.7,0.08,0.6} 
 
\usepackage{enumerate}

\theoremstyle{plain}
\newtheorem{Thm}{Theorem}
\newtheorem{Cor}[Thm]{Corollary}
\newtheorem{Prop}[Thm]{Proposition}
\newtheorem{Lem}[Thm]{Lemma}

\theoremstyle{definition}

\newcommand{\Dm}{D}

\newcommand{\rg}{\rho_\gamma}

\newcommand{\Cino}{\mathbf C}

\newcommand{\Do}{D}
\renewcommand{\Do}{D}

\newcommand{\ro}{\rho}
\newcommand{\rt}{\rho}

\renewcommand{\r}{\rho}

\renewcommand{\Re}{\operatorname{Re}}
\renewcommand{\bar}{\overline}
\renewcommand{\tilde}{\widetilde}
\newcommand{\C}{\mathbb C}

\newcommand{\Cin}{\mathbf C}

\newcommand{\Ctr}{\mathcal C}

\newcommand{\Ein}{\mathbf{E}}

\renewcommand{\iota}{s}

\newcommand{\Rin}{\mathbf{R}}

\newcommand{\N}{\mathbb N}

\newcommand{\dee}{\partial}

\newcommand{\deebar}{\overline\partial}

\renewcommand{\N}{\nu}
\newcommand{\bndry}{b}

\begin{document}
\title
[Cauchy-Leray Integrals]{The role of an integration identity\\
in the analysis of \\
the Cauchy-Leray Transform}
\author[Lanzani and Stein]{Loredana Lanzani$^*$
and Elias M. Stein$^{*}$}
\thanks{$^*$ This material is based upon work supported in part by the National Science Foundation under awards No.
DMS-1503612 (Lanzani) and DMS-1265524 (Stein), while both authors were in residence at the Mathematical Sciences Research Institute in Berkeley, California, during the Spring 2017 semester.}
\address{
Dept. of Mathematics,       
Syracuse University 
Syracuse, NY 13244-1150 USA}
  \email{llanzani@syr.edu}
\address{
Dept. of Mathematics\\Princeton University 
\\Princeton, NJ   08544-100 USA }
\email{stein@math.princeton.edu}
\begin{abstract} 

The purpose of this paper is to complement the results in \cite{LS-1} by showing the dense definability 
of the Cauchy-Leray transform for the domains that give the counterexamples of \cite{LS-1}, where 
$L^p$-boundedness is shown to fail when either the ``near'' $C^2$ boundary regularity, or the
strong $\mathbb C$-linear convexity assumption is dropped.
\end{abstract}
\maketitle
\centerline{\em Dedicated to the memory of Professor Minde Cheng}
\centerline{\em  on the occasion of the centenary of his birth}

\section{Introduction}
The purpose of this paper is to complement the previous results in \cite{LS-1} and \cite{LS-2} which deal with the $L^p$-boundedness of the Cauchy-Leray integral in $\mathbb C^n$, $n\geq 2$, proved under optimal regularity assumptions and geometric restrictions. 

If $D$ is a suitable (convex) domain in $\mathbb C^n$ then we can define the Cauchy-Leray integral \cite{LS-2} $\Cin (f)$ for an appropriate function $f$ given on the boundary $\bndry D$ of $D$. The integral $\Cin (f)(z)$ is defined for $z\in D$ and has the following properties:
 first, $\Cin (f)(z)$ is always holomorphic for $z\in D$; and second, if $f=F\big|_{\bndry D}$ where $F$ is continuous in $\bar D$ and holomorphic in $D$, then $\Cin$ reproduces $f$, i.e. $\Cin(f)(z) = F(z)$, for $z\in D$.

The well-known theory of the Cauchy integral in $\mathbb C$ (see \cite{C}, \cite{CMM}, \cite{Da}) and in particular the classical theorem of M. Riesz for the unit disc, raise the question of the corresponding $L^p$-boundedness of the Cauchy-Leray integral for $n\geq 2$, and more particularly determining the optimal conditions both on the regularity of the domain $D$ and the nature of the convexity of $D$, for which this boundedness holds. To be precise, suppose $f$ is a 
$C^1$ function on the boundary of $D$, then under the circumstances detailed below
we can define the induced ``Cauchy-Leray transform'' of $f$, $\Ctr(f)$, as a function on the boundary of $D$ given by
\begin{equation}\label{E:star1}
\Ctr (f) =\Cin (f)\bigg|_{\bndry D}\, ,
\end{equation}
and it is proved in \cite{LS-2} that the mapping $f\mapsto \Ctr (f)$ extends to a bounded transformation on $L^p(\bndry D, d\sigma)$, $1<p<\infty$, where $d\sigma$ is the induced Lebesgue measure on $\bndry D$. This assertion holds under the following two conditions:
\begin{itemize}
\item[]
\item[{\tt (i)}] The boundary has regularity ``near'' $C^2$, in the sense that $\bndry D$ is of class $C^{1,1}$.
\item[]
\item[{\tt (ii)}] The boundary is ``strongly $\mathbb C$-linearly convex''.
\item[]
\end{itemize}

The convexity condition {\tt (ii)} is weaker than strong convexity (the strict positive-definiteness of the real quadratic fundamental form at each boundary point); however it is stronger than strong-pseudoconvexity (the strict positivity of the Levi form).

In \cite{LS-1} two simple counter-examples, elucidating the necessary nature of both conditions were given. These are in terms of two elementary domains in $\mathbb C^2$. 
With $z=(z_1, z_2)$, and $z_1=x_1+\, i\,y_1$ they are given by 
\begin{equation}\label{E:star2}
\{z\ :\ |z_2|^2 + x_1^2+y_1^4<1\}
\end{equation}

and

\begin{equation}\label{E:star3}
\{z\ :\ |z_2|^2 + |x_1|^m+y_1^2<1\}\, \quad \text{where}\quad 1<m<2\, .
\end{equation}
\medskip

Equivalently, we could replace $z_2$ by $z_2-i$, as we did in \cite{LS-1}.
The first domain \eqref{E:star2} has a $C^\infty$ (in fact real-analytic) boundary, is strongly pseudo-convex but not strongly $\mathbb C$-linearly convex. The second domain \eqref{E:star3} is of class $C^{2-\epsilon}$, with $\epsilon = 2-m$, but on the other hand, is strongly $\mathbb C$-linearly convex.

For both domains we proved in \cite{LS-1} that $L^p$-boundedness failed for all $p$,
 $1\leq p\leq \infty$, in the following sense. Whenever $f$ is a bounded function on $\bndry D$ supported on a proper subset of the boundary, $\Cin (f)(z)$ can be defined (as an absolutely convergent integral) whenever $z\in\bndry D$ is at a positive distance from the support of $f$, (which we still denote by $\Cin (f)(z)$). One might ask if there is an inequality of the form 
 
 \begin{equation}\label{E:star4}
 \|\Cin (f)\|_{L^p(S)}\ \leq A_p\|f\|_{L^p(\bndry D)}
 \end{equation}

whenever $S$ is a subset of $\bndry D$ disjoint from the support of $f$, with the bound in \eqref{E:star4} independent of $f$ or $S$, and with the underlying measure the induced Lebesgue measure. This was shown in \cite{LS-1} to fail  for both domains for all $p$, $1\leq p\leq \infty$.

In order to define the Cauchy-Leray transform for the domains that satisfy {\tt(i)} and {\tt(ii)}, it was shown in \cite{LS-2} that whenever $f$ is of class $C^1$ on the boundary, then
$\Cin (f)(z)$ extends to a continuous function on $\bar D$. It is our purpose here to demostrate that a similar assertion holds for the counter-example domains \eqref{E:star2}
and \eqref{E:star3} and thus the induced Cauchy-Leray transform $f\mapsto \Ctr (f)$, initially defined for $C^1$-functions, does not extend to a bounded operator on $L^p$, for any $p$, for both domains \eqref{E:star2} and \eqref{E:star3}.

Restricting ourselves to the case $n=2$, one can write the Cauchy-Leray integral as

\begin{equation}\label{E:star5}
\Cin (f) (z) \ =\ \int\limits_{w\in\bndry D}\!\!\!\! \frac{f(w)}{(\Delta (w, z))^2}\, d\lambda (w)
\end{equation}

where $\Delta (w, z)$ is the non-hermitian pairing 
$$\langle \dee\r (w), w-z\rangle\ =\ \sum\limits_{j=1}^2\frac{\dee\r (w)}{\dee w_j}\, (w_j-z_j),$$ with $\r$
a defining function of $D$, and $d\lambda$ the corresponding Leray-Levi measure.

For the domains that satisfy {\tt (i)} and {\tt (ii)}, the continuity of $\Cin (f)(z)$, $z\in\bar D$ when $f$ is of class $C^1$ on $\bndry D$ can be shown as a consequence of the principle that the Cauchy-Leray kernel (when $n\geq 2$) is a ``derivative''. This is most aptly expressed in the global integration-by-parts performed in the proof of Proposition \ref{P:Dm-case1} below, which can be summarized
as
\begin{equation}\label{E:star6}
\Cin (f)  = \Ein (df) + \Rin (f) + f
\end{equation}

Here the kernels of the operators $\Ein$ and $\Rin$ both have singularities weaker than 
that of $\Cin$ and are in fact absolutely integrable singularities. The one for $\Ein$ is of the order

\begin{equation}\label{E:sing1}
\frac{1}{|\Delta (w, z)|}\, ,
\end{equation}
\smallskip

and the one for $\Rin$ is of the order

\begin{equation}\label{E:sing2}
\frac{|z-w|}{|\Delta (w, z)|^2}\, ,\quad \text{or even}\qquad \frac{|z-w|^2}{|\Delta (w, z)|^2}\, .
\end{equation}

However application of this formalism to the first example, \eqref{E:star2}, is problematic: because of the ``flatness'' of that domain (which can be restated in terms of the higher degree of vanishing of $\Delta(w, z)$  along the diagonal $\{w=z\}$, see \eqref{E:Do-Stricvx}),
the integrability of \eqref{E:sing2} cannot be guaranteed because the numerator in \eqref{E:sing2} does not help to control the  singularities that are away from the diagonal. Thus in this case a different argument is needed, one that uses a local integration by parts, depending on the location of the coordinate patch with respect to the
``flat'' part of the boundary. This is carried out in Proposition \ref{P:CintDo} in Section \ref{S:2} below.

In the second example the difficulty comes from the lack of $C^2$-regularity. Here, however, a simple but critical modification of the argument that was used in \cite{LS-2} to prove an earlier version of \eqref{E:star6} for the ``nice ''domains  works for the most relevant choices of $m$ that pertain the domain \eqref{E:star3}.
 This is carried out in Proposition \ref{P:Dm-case1} and Lemma \ref{L:estimates} in Section \ref{S:3} below.
 
 These results together with what was done in \cite{LS-2} then give us our main conclusion:
 
 \begin{Thm}
 In the cases of the domain \eqref{E:star2}, and the domain \eqref{E:star3} for 
 $3/2<m<2$, for any function $f$ of class $C^1$ on $\bndry D$, the Cauchy integral $\Cin (f)(z)$ extends to a continuous function on $\bar D$. If we set
 $$
 \Ctr (f) \ =\ \Cin (f)\bigg|_{\bndry D}
 $$
 then the mapping $f\mapsto \Ctr (f)$ cannot be extended to a bounded operator on $L^p(\bndry D)$ for any $p$, $1\leq p\leq \infty$.
 \end{Thm}

There are two additional comments to make. First, since the main interest of the second example is the class $C^{2-\epsilon}$ with $\epsilon$ small, this clearly falls within the restriction $3/2<m<2$. Second, there is a weaker conclusion that covers the full range
$1<m<2$. This is stated in Proposition \ref{P:Dm-case2} below.

\section{The first example}\label{S:2}
Writing $z_j=x_j+i\,y_j$ for $j=1, 2$, our first example is the domain \eqref{E:star2}, that is

\begin{equation*}
\Do \ =\ \{(z_1, z_2)\ :\ |z_2|^2 + x_1^2 + y_1^4\ <\ 1\, \}.
\end{equation*}
\subsection{The Cauchy-Leray integral for the domain \eqref{E:star2}}\label{SS:Do-4} Recall that $\Do$ is strictly convex and this grants
\begin{equation}\label{E:Do-Stricvx}
(\dee\ro (w), w-z)_{\mathbb R} \geq 
\end{equation}
\begin{equation*}
\geq\ (x_1-u_1)^2 + (x_2-u_2)^2 + (y_2-v_2)^2 + (v_1^2+y_1^2)(v_1-y_1)^2
\end{equation*}

\noindent when $w, z\in\bndry \Do$.
From this strict convexity it follows that $\Do$ supports the Cauchy-Leray
integral
\begin{equation}\label{E:Cin-o}
\Cino f(z) =\frac{1}{(2\pi i)^2}\,\int\limits_{w\in\bndry \Do}\!\!\!\!\! f(w)\,j^*\frac{(\dee\ro\wedge\deebar\dee\ro)(w)}
{\langle\dee\ro(w), w-z\rangle^2}
\, , \quad z\in \Do\, ,
\end{equation}
where $j^*$ denotes the pullback via the inclusion $j\!:\bndry \Do\hookrightarrow \C^2$,
and in fact with respect to the induced Lebesgue measure $d\sigma$, we have
\begin{equation*}
\Cino f(z) =
\int\limits_{w\in\bndry \Do}\!\!\!\!\! \frac{f(w)\,
\Gamma (w)}
{\langle\dee\ro(w), w-z\rangle^2}\, d\sigma (w)
\, , \quad z\in \Do\, ,
\end{equation*}
with $\Gamma$ a smooth, strictly positive function on $\bndry D$, see \cite{LS-1}.
\begin{Prop}\label{P:CintDo}
Suppose that $f\in C^1(\bndry\Do)$. Then $\Cino f(z)$ extends to a continuous function on $\bar \Do$.
\end{Prop}
This proposition allows us to define the induced Cauchy-Leray transform $\Ctr (f)$, at least
initially for $f$ that are of class $C^1$ on $\bndry \Do$, by
\begin{equation}\label{E:Ctro-lim}
\Ctr f \ =\ \Cino f\big|_{\bndry \Do}\, .
\end{equation}
\subsection{Proof of Proposition \ref{P:CintDo}}
 Throughout this section we shall simplify the notation slightly by writing $\rt$ instead of $\ro$, and by setting 
  $$
  \Delta (w, z) = \langle \dee\rt (w), w-z\rangle\, .
  $$
  
  We need first a local representation of the Cauchy-Leray integral, which was the idea that 
  $\Delta (w, z)$ arises as a ``derivative''. 
  
  We fix a point $\zeta\in\bndry\Do$, and we will restrict attention to points $z\in\bar\Do$, $w\in\bndry\Do$ both near $\zeta$. We can introduce a new coordinate system, centered at $\zeta$, so that we pass from the original
   coordinates in $\mathbb C^2$ to the present coordinates by a translation and a unitary
  linear transformation, and so that in the new coordinates $w=(w_1, w_2)$, $w_j=u_j+iv_j$, the point $\zeta$ has coordinates $(0, 0)$ and 
  $$
  \frac{\dee\rho}{\dee v_2}(\zeta) = |\nabla \rho (\zeta)|\, ,\ \text{while}\quad
  \frac{\dee\rho}{\dee u_1}(\zeta)= \frac{\dee\rho}{\dee v_1}(\zeta)= \frac{\dee\rho}{\dee u_2}(\zeta)=0.
  $$
Since then the tangent space to $\bndry\Do$ at $\zeta$ is given by $\{(u_1+iv_1, u_2)\}$ we can
express $\bndry\Do$ near $\zeta$ as a graph $v_2=\Phi (u_1, v_1, u_2)$, with $\Phi (0, 0, 0)=0$ and
$\nabla\Phi(0, 0, 0)=(0, 0, 0)$. In particular if $w=(w_1, w_2)\in\bndry \Do$ is close to $\zeta$ with say, 
$|w-\zeta|<\delta$ then
$$
\frac{\dee w_2}{\dee u_2}\ =\ 1 +\ i\ \frac{\dee \Phi}{\dee u_2}\ =\ 1+i\, O(\delta).
$$

\begin{Lem} We have
\begin{equation}\label{E:7}
\frac{\dee \Delta}{\dee u_2} (w, z) \neq 0\, ,\ \text{if}\ \ |w-\zeta|<\delta\, ,\ \text{and}\ \ |z-\zeta|<\delta
\end{equation}
when $\delta$ is sufficiently small.
\end{Lem}
\begin{proof}
$$
\frac{\dee \Delta}{\dee u_2} (w, z) \ =\ I\ +\ II\ +\ III
$$
where
$$
 I\ =\ \bigg\langle\frac{\dee}{\dee u_2}\dee\rho (w),\, w-z\bigg\rangle
$$
and thus $|\, I\, |\leq c\, \delta$\  because $|w-z|\leq 2\delta$;
and
$$
II\ =\ \bigg\langle\dee\rho (w)-\dee\rho (\zeta), \frac{\dee w}{\dee u_2} \bigg\rangle
$$
and therefore $|\, II\, |\leq c\, \delta$\  because $|w-\zeta|\leq \delta$.
However
$$
III = \bigg\langle\dee\rho (\zeta), \frac{\dee w}{\dee u_2}\bigg\rangle\ =\ 
\bigg[0,\frac12\left(\frac{\dee\rho}{\dee u_2}(\zeta)-i \frac{\dee\rho}{\dee v_2}(\zeta)\right)\bigg]\cdot \bigg[0, \big(1+ i\,O(\delta)\big)\bigg] =
$$
$$
=\ -i \,|\nabla\rho (\zeta)|\, \big(1+ i\,O(\delta)\big)\, =\ -i\,\frac{|\nabla\rho (\zeta)|}{2} + O(\delta).
$$

Thus $|\, III\, | \geq |\nabla\rho (\zeta)|>c_1>0$ and so \eqref{E:7} is established.
\end{proof}

For $w$ in the above neighbrhood of $\zeta$, let us write 
$$
\Gamma(w)\, d\sigma (w)\
 = \Lambda (u_1, v_1, u_2)\, du_1\, dv_1\, du_2\, ,\ \ \text{and}\ \ \gamma(w) =
\displaystyle{\frac{\Lambda (w)}{\frac{\dee \Delta}{\dee u_2}(w, z)}}\, .
$$
Then $\gamma (w)$ is a smooth function (uniformly in $z$ for $|z-\zeta|<\delta$). 

\begin{Cor}\label{C:E8}
 Suppose $f$ is a $C^1$ function on $\bndry\Do$ supported in the above neighborhood of 
$\zeta$. If $|z-\zeta|<\delta$, $z\in\Do$, then
\begin{equation}\label{E:8}
\Cino f (z) \ =\ \int\limits_{\bndry\Do}\!
\frac{1}{\Delta (w, z)}\, \frac{\dee}{\dee u_2}\!\bigg(f(w)\cdot \gamma (w)\bigg) du_1\, dv_1\, du_2\, .
\end{equation}
\end{Cor}
\begin{proof}
$$
\Cino f (z) \ =\ \int\limits_{\bndry\Do}\frac{1}{\Delta^2 (w, z)}\, f(w)\, \Gamma (w)\, d\sigma(w)
$$
Then in the support of $f$, we have 
$$
\frac{1}{\Delta^2 (w, z)}\ =\ -\,\frac{1}{\displaystyle{\frac{\dee \Delta}{\dee u_2}(w, z)}}\cdot \frac{\dee}{\dee u_2}\left(\frac{1}{\Delta (w, z)}\right)\, .
$$

Inserting this in the above and carrying out the indicated integration by parts in the $u_2$ variable 
then yields \eqref{E:8}.
\end{proof}
To continue we make several other observations.
\medskip

First, whenever $z\in\bndry \Do$, write $z_\epsilon = z+\epsilon\, \N (z)$, where $\N (z)$ is the inward unit normal vector at $z$ and $\epsilon \geq0 $, that is
$$
\N(z) = \ -\, \frac{\nabla\r (z)}{|\nabla\r (z)|}.
$$

   Assume as before that $w, z\in\bndry \Do$, 
$|z-w|\leq 2\delta$. Then with $\delta$ sufficiently small
\begin{equation}\label{E:9}
\Re \Delta (w, z_\epsilon) \ \geq\  c\epsilon + \Re \Delta (w, z)\,\quad \text{with }\ c>0\, .
\end{equation}
In fact 
\begin{equation}\label{E:9a}
\Delta (w, z_\epsilon) =\Delta (w, z) -\epsilon \langle \dee\rho (w), \N (z)\rangle\, .
\end{equation}

But $\N (z) = \ \N (w) + O(|z-w|)$, thus
$$
\langle \dee\rho (w), \N (z)\rangle \ =\ \langle \dee\rho (w), \N (w)\rangle \ +\ O(|z-w|)\, ,
$$
and altogether 
$$
\Delta( w, z_\epsilon) - \Delta (w, z) \ =\ -\epsilon\, \langle \dee\rho (w) , \N (w)\rangle + O(\epsilon |z-w|)\, .
$$
Since $\Re \langle \dee\rho (w) , \N (w)\rangle = -\, |\nabla \rho (w)\, |$, this yields \eqref{E:9}. 

In fact, if $|w-z|>2\delta$ then the strict convexity of $\Do$ gives that $\Re \Delta (w, z_\epsilon)>\tilde c$\,; by the compactness of $\bar \Do$ if we take $0<c_0<1$ sufficiently small, we may combine  all of the above and conclude that
\begin{equation}\label{E:6a}
\Re \Delta (w, z_\epsilon)\geq c_0\bigg(\epsilon + \Re \Delta (w, z)\bigg)\quad \text{uniformly\ in}\ \ z, w\in\bndry\Do.
\end{equation}
\medskip

The key observation is that 
 \begin{equation}\label{E:10}
 \int\limits_{\bndry\Do}\frac{1}{\left(\Re \Delta (w, z)\right)^{1+\beta}}\, d\sigma (w)\ \leq C_\beta
 \end{equation}
uniformly for $z\in\bndry\Do$, if $0\leq \beta<1/4$, and with $\sigma$ the induced Lebesgue measure
on $\bndry\Do$.

For small fixed $\eta>0$, we may assume that the integration in \eqref{E:10} is over the $\eta$-ball centered at $z$; otherwise the strict convexity of $\Do$ gives that 
$\Re \Delta (w, z)\geq c>0$ if $|w-z|\geq \eta$, and  the compactness of $\bar\Do$ grant that the inequality 
\eqref{E:10} holds when the integration is taken over the complement of the $\eta$-ball about $z$.

To illustrate what is done next, assume that $z$ lies in the ball centered at the origin. Then in the coordinates $z= (z_1, z_2)$ and $w=(w_1, w_2)$, we know by \eqref{E:Do-Stricvx} that in particular
$$
2\Re \Delta (w, z) \geq (x_1-u_1)^2 + (x_2-u_2)^2 + v_1^2 (v_1-y_1)^2\, .
$$

Moreover the tangent space to $\bndry\Do$ at the origin is given by $(u_1+iv_1, u_2)$. So far $w$ is near the origin, $\bndry \Do$ is represented by a graph $(u_1+iv_1, u_2+i\Phi (u_1, v_1, u_2))$ with $\Phi$ a smooth function, and hence the induced measure $d\sigma$ is given by $d\sigma = D(u_1, v_1, u_2)\, du_1\,dv_1\,du_2$, where the density $D$ is bounded. This shows that \eqref{E:10} will be proved as soon as we have that 
\begin{equation}\label{E:11}
\int\limits_B\!\frac{1}{\big(u_1^2+u_2^2 + v_1^2(v_1-y_1)^2\big)^{1+\beta}}\, du_1\, du_2\, dv_1\, <\infty
\ \  \text{uniformly\ in}\ \ y_1
\end{equation}
where $B$ is the unit ball: $\{u_1^2 + v_1^2 + u_2^2<1\}$ in the parameter space.

However 
$$
\int\limits_{\mathbb R^2}\frac{1}{\big(u_1^2 + u_2^2 + A\big)^{1+\beta}}\ du_1\, du_2\ =\
 C_\beta A^{-\beta},\  A>0,\ \ \beta>0
$$
as a simple rescaling $u_1= u_1'A^{1/2}$, $u_2= u_2'A^{1/2}$ shows. So we take 
$A = \big(v_1(v_1-y_1)\big)^2$, and observe that 
$$
\int\limits_{|v_1|\leq 1}\frac{1}{|v_1|^{2\beta}\, |v_1-y_1|^{2\beta}}\, dv_1\, <\infty\, ,\ \ \text{if}\ \ 0\leq \beta<1/4\, ,
$$
and this proves \eqref{E:11} for $0<\beta<1/4$ and thus also for $\beta =0$.
\medskip

We next lift the restriction that $z$ lies in the $\eta$-neighborhood of the origin, and assume instead that 
$z$ lies in the $\eta$-neighborhood of $\zeta$ , for some fixed $\zeta\in\bndry \Do$. Now the tangent
space to $\bndry\Do$ at $\zeta$ is given by $\{w\ :\ (\N (\zeta), w-\zeta)_{\mathbb R}=0\}$, where 
$\N (\zeta)$ is the inner unit normal at $\zeta$. 
\smallskip

\noindent We consider the four real component of $\N (\zeta)$, that we list as $\nu_1, \nu_1', \nu_2, \nu_2'$, which correspond to the the $u_1, v_1, u_2, v_2$ variables, respectively. Denote by $\bar\nu$ a component among these four that has the maximum absolute value. Now there are two cases:
{\em Case 1}:\ \  $\bar\nu\neq \nu_1'$;\ \  {\em Case 2}:\  \ $\bar\nu = \nu_1'$.

Now in {\em Case 1} (which is what happens at $\zeta =0$, since there $\bar\nu = \nu_2 =1$), assume momentarily that $\bar \nu =\nu_1$. Then the tangent space at $\zeta = (\xi_1+i\eta_1, \xi_2 + i\eta_2)$ can be written as 
$$
u_1 -\xi_1 = a(v_1 -\eta_1) +b(u_2-\xi_2) + c(v_2-\eta_2)
$$
with $|a|,\ |b|,\ |c|\leq 1$. So if we take $v_1$, $u_2$ and $v_2$  as independent variables to represent 
$\bndry\Do$ as a graph, we see by \eqref{E:Do-Stricvx} that it suffices to show that
$$
\int\limits_{B'}\!\frac{1}{\big(v_2^2+u_2^2 + v_1^2(v_1-y_1)^2\big)^{1+\beta}}\, dv_2\, du_2\, dv_1\, <\infty
\ \  \text{uniformly\ in}\ \ y_1\, ,
$$
which is the same as \eqref{E:11}, except that $u_1$ has been replaced by $v_2$.
Similarly if $\bar\nu = \nu_2$, or $\bar \nu = \nu_2'$.

Now in {\em Case 2}, when $\bar \nu =\nu_1'$, we represent the tangent space as 
$$
v_2-\eta_2 = 
a'(u_1 -\xi_1) +b'(u_2-\xi_2) + c'(v_2-\eta_2)
$$
and then by \eqref{E:Do-Stricvx} it suffices to see that
$$
\int\limits_{B'}\!\frac{1}{\big(u_1^2+u_2^2 + v_2^2\big)^{1+\beta}}\, du_1\, du_2\, dv_2\, <\infty
$$
which in fact holds for $0\leq \beta<1$. This concludes the proof of \eqref{E:10}.

Returning to the proof of Proposition \ref{P:CintDo}, we will show that whenever $f$ is a $C^1$ function on $\bndry\Do$, then $\Cino f(z_\epsilon)$ converges uniformly for $z\in\bndry\Do$, as $\epsilon\to 0$.
To see this we decompose $f$ as a finite sum 
$$
f = \sum \limits_{j=1}^N f_j\, ,
$$
where each $f_j$ is a $C^1$ function supported in a ball $B_j$ of radius $\eta/2$, centered at
$\zeta^{(j)}\in\bndry\Do$. If $z\in B^*_j$ (the ball with the same center but twice the radius) and 
$\eta$ is sufficiently small, then by Corollary \ref{C:E8}
\begin{equation}\label{E:12}
\Cino f_j(z_\epsilon) \ =\ 
\int\limits_{\bndry\Do}\frac{1}{\Delta (w, z_\epsilon)}\, \frac{\dee}{\dee u_2}(f_j\gamma)\, du_1\, dv_1\, du_2\, .
\end{equation}
Next observe that 
\begin{equation}\label{E:13}
\int\limits_{\bndry\Do}\left|\frac{1}{\Delta (w, z_\epsilon)} - \frac{1}{\Delta (w, z)}\right|\, d\sigma (w)\ \lesssim
 \epsilon^\beta\, .
\end{equation}
Indeed because both $\Re \Delta (w, z_\epsilon)$ and $\Re \Delta (w, z)$ are non-negative and
$\Delta (w, z_\epsilon) - \Delta (w, z) = O(\epsilon)$, see \eqref{E:9a}, the integrand in \eqref{E:13} is dominated
by 
$$
\frac{c\,\epsilon}{\Re \Delta (w, z_\epsilon)\, \Re \Delta (w, z)}\, .
$$

But \eqref{E:9a} tells us more precisely that
$$
\Re \Delta (w, z_\epsilon)\ \geq \ c_0\Re \Delta (w, z)\quad \text{and}\ \ \Re \Delta (w, z_\epsilon)\geq c_0\, \epsilon\, .
$$
As a result, the integrand in \eqref{E:13}  is dominated by a multiple of
$$
\frac{\epsilon^\beta}{(\Re \Delta (w, z))^{1+\beta}},
$$
and we need only invoke \eqref{E:10} to get \eqref{E:13}.
From this and \eqref{E:12} we obtain that $\Cino f_j(z_\epsilon)$ converges uniformly for 
$z\in B^*_j$ as $\epsilon \to 0$. However when $z\notin B^*_j$, then $\Delta (w, z_\epsilon)\neq 0$
 for the relevant $w$ and $z$, and the convergence as $\epsilon\to 0$ is obvious. This gives the desired result for each $f_j$, and hence for their sum, proving Proposition \ref{P:CintDo}.
\section{The second example}\label{S:3} For our second example, the domain \eqref{E:star3},
we choose $$\r (w) = |w_2|^2 + |u_1|^m + v_1^2 -1$$ as a defining function.
Here and in the sequel we make use of the notation: 
$$[u_1] = |u_1|^{m-1}\text{sign} (u_1)\, .$$

The Cauchy-Leray denominator $\Delta (w, z) =\langle \dee\r (w), w-z\rangle$ for $\Dm$
 is then
$$
\Delta (w, z) \ =\ 
$$
$$
\frac12 \big(m[u_1]^{m-1} + 2i  + 2 v_1\big)\big(u_1-v_1 + i(v_1-y_1)\big)\  +\ \big(u_2 -i\, v_2\big)\big(u_2 - x_2 + i(v_2-y_2)\big).
$$
\smallskip

As a preliminary step we decompose the boundary of $\Dm$ into finitely many coordinate patches, either of the first or the second kind. The coordinate patches of the first kind 
are centered at points which lie in the critical variety (where $u_1=0$). Those of the second kind are at a positive distance from that variety. We then decompose our given $C^1$-function $f$ as a sum of $C^1$ functions, each supported on one of these patches. 
For those of the second type, since we are now where matters are regular, we may argue as in Section \ref{S:2}. This reduces matters to the patches of the first kind. Since for these patches we have that $u_1$ is small, then one of the three variables $v_1, u_2, v_2$ must be bounded away from zero. Whichever is can be taken as the dependent variable in the representation of $\bndry\Dm$ as a graph over that patch. For simplicity of notation here we 
assume it is $v_1$, but if it were $u_2$ or $v_2$ instead, then the argument below would be unchanged. Then for $v_1$ (the dependent variable) we have

\begin{equation}\label{E:v_1}
v_1 \ =\ \big(1-(|u_1|^m +|w_2|^2)\big)^{1/2}\ \approx\  1-\frac12 \bigg(|u_1|^m + |w_2|^2\bigg)\, .
\end{equation}
The following basic estimates for $\Delta (w, z)$ can be proved as in \cite[Lemma 4.3]{LS-2}:
$$
 \Re \Delta (w, z_\epsilon)\ -\Re \Delta (w, z)\ \gtrsim\  \epsilon\, 
 $$
where as before, $z_\epsilon = z+\epsilon \N (z)$, and
$$
 |\Delta(w, z_\epsilon) -\Delta(w, z)|\ \lesssim\ \epsilon\, 
$$
for any $w$, $z\in\bndry D$.
\subsection{The Cauchy-Leray integral for the domain the domain \eqref{E:star3}.}
Suppose that $f\in C^1(\bndry\Dm)$ and set
\begin{equation}\label{E:CL-Dm-old}
\Cin f(z) = \int\limits_{w\in\bndry\Dm} \frac{f(w)}{\Delta^2 (w, z)}\, d\lambda (w)\, \quad z\in \Dm
\end{equation}
with $d\lambda (w) = j^*(\dee\r (w)\wedge \deebar\dee\r (w))$.
Then in fact
$$
\Cin (f)(z) = \int\limits_{w\in\bndry\Dm}\!\!\!\!\! f(w)\, j^*\bigg(\!\eta(w, z)\wedge \deebar\, \eta (w, z)\bigg)
$$
where
$$
\eta(w, z)\ :=\ \frac{\dee\r (w)}{\Delta (w, z)}
$$
is a so-called {\em generating form for $\Dm$}, see \cite[Sections 4.1 and 9.2]{LS-3}. The Cauchy-Fantappi\`e theory then grants that
$$
\Cin (g)(z) = g(z)\,,\quad z\in \Dm
$$
whenever $g$ is holomorphic in $\Dm$ and continuous on $\bar\Dm$, 
see \cite[Section 5]{LS-3}. In particular choosing $g(z) =1$ (the constant 1) we have that
$$
\Cin (1) (z) = 1\, \quad \text{for}\ z\in \Dm\quad \text{and\ thus\ for}\ \ z\in\bar\Dm\, ,
$$
and from this it follows that 
\begin{equation}\label{E:CL-Dm}
\Cin (f)(z) = \int\limits_{w\in\bndry\Dm}\!\!\! \frac{f(w)-f(z)}{\Delta^2 (w, z)}\, d\lambda (w)\, +\ f(z)\, ,\quad z\in \Dm\, .
\end{equation}
\begin{Prop}\label{P:Dm-case1} Let $\Dm$ be the domain \eqref{E:star3} with $3/2<m<2$. 
 Suppose that $f\in C^1(\bndry\Dm)$. Then $\Cin (f)$ extends to a continuous function in $\bar\Dm$. More precisely, for 
$z\in\bndry\Dm$ set
$$z_\epsilon = z +\epsilon\, \N (z)\ \in \Dm\, .$$ 
Then, we have that  $\Cin (f) (z_\epsilon)$ converges uniformly in
 $z\in \bndry\Dm$ to a limit that we denote $\Ctr (f)(z)$. Furthermore, we have that
 such limit admits the representation
 $$
 \Ctr (f)(z)\ =\ \Ein (df) (z)\ +\ \Rin (f)(z)\ +\ f(z)\, , \quad z\in\bndry\Dm
 $$
 where $\Ein (df)$ and\ \ $\Rin (f)$ are absolutely convergent integarls given explicitly by \eqref{E:Ess} and \eqref{E:Rem} below.
\end{Prop}
\begin{proof}
By the above considerations we have that 
$$
\Cin (f)(z_\epsilon) = \int\limits_{w\in\bndry\Dm}\!\!\! \frac{f(w)-f(z)}{\Delta^2 (w, z_\epsilon)}\, d\lambda (w)\, +\ f(z)\quad\text{for\ any}\ \ z\in\bndry\Dm\, .
$$
We will show that 
the quantity
$$
\int\limits_{w\in\bndry\Dm}\!\!\! \frac{f(w)-f(z)}{\Delta^2 (w, z_\epsilon)}\, d\lambda (w)
$$
converges 
 uniformly in $z\in\bndry\Dm$ as $\epsilon\to 0$. To this end, we consider the following
 smooth approximation of $\r$, see \cite[Section 5]{LS-1}:
 \begin{equation}\label{E:rg}
\rg (w) = (u_1^2+\gamma)^{m/2} + v_1^2 + u_2^2 +v_2^2 -1\, .
\end{equation}
 Then  for any $0<\gamma<1$ we have
 
 $$
\int\limits_{w\in\bndry\Dm}\!\!\! \frac{f(w)-f(z)}{\Delta^2 (w, z_\epsilon)}\, d\lambda (w)\ =\  A_\gamma (z_\epsilon)\ +\ B_\gamma (z_\epsilon)
 $$
 with
 \begin{equation}\label{E:Agamma}
 A_\gamma (z_\epsilon) =\ 
 -\frac{1}{4\pi^2}\!\!\! \!\!\! 
 \int\limits_{w\in\bndry\Dm}\!\!\! \frac{f(w)-f(z)}{\Delta^2 (w, z_\epsilon)}\, 
 j^*\big(\dee\r \wedge\deebar\dee(\r-\rg )\big)(w)\,.
 \end{equation}
and
 \begin{equation}\label{E:Bgamma}
 B_\gamma (z_\epsilon) =\ -\frac{1}{4\pi^2}\!\!\! \!\!\! 
 \int\limits_{w\in\bndry\Dm}\!\!\! \frac{f(w)-f(z)}{\Delta^2 (w, z_\epsilon)}\, 
 j^*\big(\dee\r (w)\wedge\deebar\dee\rg (w)\big)\, .
 \end{equation}
 The dominated convergence theorem grants that 
  \begin{equation}
  A_\gamma (z_\epsilon)\ \to\ 0\quad \text{as}\ \ \gamma\to 0\, 
  \end{equation}
  for each fixed $z_\epsilon$.
 Next we deal with the term $B_\gamma (z_\epsilon)$, again assuming that 
 $z_\epsilon$ has been fixed. Applying Stokes' theorem to 
 the manifold $M=\bndry\Dm$ (which has $\bndry M=\emptyset$) we obtain that
 $$
 0\ =\ 
 \int\limits_{w\in\bndry\Dm}\!\!\! 
 d_w\, j^*\!\left(
 \frac{1}{\Delta (w, z_\epsilon)}\, \big(f(w)-f(z)\big)\, \deebar\dee\rg (w)
 \right)\, .
 $$
 Thus (since $d\,\deebar\dee\rg \equiv 0$) we obtain
 $$
 0= \int\limits_{w\in\bndry\Dm}\!\!\! \!\!
 \frac{1}{\Delta (w, z_\epsilon)}\, j^*(df(w)\wedge\deebar\dee\rg (w))\ +
 $$
 $$
 -\!\!\! \!\!\int\limits_{w\in\bndry\Dm}\!\!\! \frac{f(w)-f(z)}{\Delta^2 (w, z_\epsilon)}
 j^*\big[d_w(\Delta (w, z_\epsilon)\big]\wedge\deebar\dee\rg (w)\, .
 $$
 But 
 $$
 d_w\Delta (w, z_\epsilon)\ =\ \dee\r (w) + \nabla^2\r (w)\cdot (w-z)
 $$
 
 where $\nabla^2\r (w)\cdot (w-z)$ is short-hand for the 1-form
 $$
 \sum\limits_{j=1}^2\left(\frac{\dee^2\r(w)}{\dee w_j\dee w_1}\, dw_1 +
 \frac{\dee^2\r(w)}{\dee w_j\dee w_2}\, dw_2\right)(w_j-z_j)\, .
 $$
 It follows that
 $$
 B_\gamma (z_\epsilon)\, =\!\!\! 
  \int\limits_{w\in\bndry\Dm}\!\!\! 
 \frac{j^*(df(w)\wedge\deebar\dee\rg (w))}{\Delta (w, z_\epsilon)}\ +
 $$
 $$
 - \!\!\int\limits_{w\in\bndry\Dm}\!\!\! \frac{f(w)-f(z)}{\Delta^2 (w, z_\epsilon)}
\ \nabla^2\!\r (w)\!\cdot\! (w-z)\wedge j^*\deebar\dee\rg(w)\, .
 $$
Invoking  the dominated convergence theorem one more time,
we  obtain 
 $$
 \lim\limits_{\gamma\to 0}\,B_\gamma (z_\epsilon)\, =\!\!\! 
  \int\limits_{w\in\bndry\Dm}\!\!\! 
 \frac{j^*(df(w)\wedge\deebar\dee\r (w))}{\Delta (w, z_\epsilon)}\ +
 $$
 $$
 - \!\!\int\limits_{w\in\bndry\Dm}\!\!\! \!\frac{f(w)-f(z)}{\Delta^2 (w, z_\epsilon)}
\ \nabla^2\!\r (w)\!\cdot\! (w-z)\wedge j^*\deebar\dee\r(w)\, .
 $$

 Combining all of the above we conclude that
 $$
 \Cin (f)(z_\epsilon) - f(z) \ =\ \Ein (df)(z_\epsilon) + \Rin (f)(z_\epsilon) 
 $$
 where we have set
 $$
\Ein (df)(z_\epsilon)\ = \int\limits_{w\in\bndry\Dm}\!\!\! 
 \frac{df(w)\wedge j^*(\deebar\dee\r (w))}{\Delta (w, z_\epsilon)}\, ,
 $$
 and
 $$
 \Rin (f)(z_\epsilon)\ =\ -\!\!\!\!
 \int\limits_{w\in\bndry\Dm}\!\!\! \frac{f(w)-f(z)}{\Delta^2 (w, z_\epsilon)}
\ \nabla^2\!\r (w)\!\cdot\! (w-z)\wedge j^*\deebar\dee\r(w)\,.
 $$
 
 Define
 \begin{equation}\label{E:Ess}
 \Ein (df) (z):=\!\!\!\int\limits_{w\in\bndry\Dm}\!\!\! 
 \frac{df(w)\wedge j^*(\deebar\dee\r (w))}{\Delta (w, z)}\, , \quad z\in\bndry\Dm
 \end{equation}
 and
 \begin{equation}\label{E:Rem}
 \Rin (f) (z):=-\!\!\!
 \int\limits_{w\in\bndry\Dm}\!\!\! \!\frac{f(w)-f(z)}{\Delta^2 (w, z)}
\ \nabla^2\!\r (w)\!\cdot\! (w-z)\wedge j^*\deebar\dee\r(w)\, , \, z\in\bndry\Dm\, .
 \end{equation}
 We shall next see that for any $z\in\bndry\Dm$, each of $\Ein (f) (z)$ and $\Rin (f) (z)$ is an absolutely convergent integral, and this in turn will grant that
 
 \begin{equation}\label{E:ER-meaning}
 |\Ein (df) (z)|<\infty;\quad  |\Rin (f) (z)|<\infty\quad \text{for\ any}\ z\in\bndry\Dm.
 \end{equation}
 \smallskip
 
  To prove these assertions we recall that

 \begin{equation}\label{E:convex}
 \Re \Delta(w, z)\geq c|w-z|^2\quad \text{for any}\ z\in\bndry\Dm\quad \text{and}\  w\in\bndry\Dm
 \end{equation}
\smallskip
 
by the convexity of $\Dm$, and that 

 $$
 j^*(\deebar\dee\r (w))\ =\ O(|u_1|^{m-2})\quad \text{and}\quad |\nabla^2\r (w)|= O(|u_1|^{m-2}).
 $$
 \smallskip
 
 Thus, since $f$ is supported in a coordinate patch of the first kind, writing $x_1=\Re z_1$ we have
 
 $$
 |\Ein (df)(z)|\leq C\!\!\!\!\!\!\int\limits_{w\in\bndry\Dm}\!\!\!\frac{|df (w)|\, |u_1|^{m-2}}{|w-z|^2}
  d\sigma (w)
 \leq
 C\!\!\!\!\!\!\iiint\limits_{u_1^2+u_2^2+v_2^2<1}\!\frac{|u_1-x_1|^{m-2}}{u_1^2+u_2^2+v_2^2}\ \, du_1\, du_2\, dv_2.
 $$
 
 The desired finiteness of $\Ein (df)(z)$ now follows from 
  Lemma \ref{L:estimates}
 below applied with $\beta =0$ and $\alpha:= m-1\in (0, 1)$ (in fact our hypothesis that $3/2<m<2$ gives that $\alpha$ is in $(1/2, 1)$).
 \medskip
 
 Similarly we have, by \eqref{E:Rem}, \eqref{E:convex} and the fact that $ f(w)-f(z) = O(|w-z|)$, that
 
 $$
 |\Rin (f)(z)|\leq C\!\!\!\!\!\!\int\limits_{w\in\bndry\Dm}\!\!\!\frac{|f (w)|\,|u_1|^{2m-4}}{|w-z|^2}
 \ d\sigma (w)
 \leq C\!\!\!\!\!\!\iiint\limits_{u_1^2+u_2^2+v_2^2<1}\!\frac{|u_1-x_1|^{2m-4}}{u_1^2+u_2^2+v_2^2}\ \, du_1\, du_2\, dv_2.
 $$
 
 The finiteness of  $\Rin (f)(z)$ is again a consequence of
    Lemma \ref{L:estimates}
 below applied with $\beta =0$ and $\alpha:= 2m-3$ (which is in $(0, 1)$ thanks to our assumption that
 $3/2<m<2$).
 \smallskip

To conclude the proof of the proposition we are left to show that
of $\Ein(df) (z_\epsilon)$ and $\Rin(f) (z_\epsilon)$ converge respectively to $\Ein(df) (z)$ and 
$\Rin(f) (z)$ uniformly in $z\in\bndry\Dm$
 (in fact, absolutely and uniformly in $z\in\bndry\Dm$). 
 To prove the convergence of $\Ein(df) (z_\epsilon)$, note that
 \bigskip
 $$
 |\, \Ein(df) (z_\epsilon) -\Ein (df)(z)\,|\, \leq C\!\!\!
 \int\limits_{w\in\bndry\Dm}\!\!\!\!|u_1|^{m-2}
 \left|\frac{1}{\Delta (w, z_\epsilon)}-\frac{1}{\Delta (w, z)}\right| d\sigma(w)\, .
 $$
 Now by the basic estimate for $\Delta (w, z_\epsilon)$ we have that
 
 $$
 \left|\frac{1}{\Delta (w, z_\epsilon)}-\frac{1}{\Delta (w, z)}\right|\leq \
 \frac{c_0\,\epsilon}{\Re \Delta (w, z_\epsilon)\, \Re \Delta (w, z)}\, .
 $$
 \smallskip
 
 And since the basic estimate for $\Re\Delta (w, z_\epsilon)$ in particular  gives $\epsilon \lesssim \Re \Delta (w, z_\epsilon)$, we also have that

 $$
 \epsilon^{1-\beta}\lesssim \Re \Delta (w, z_\epsilon)^{1-\beta}\quad \text{for\ any}\  \
 0\leq \beta<1\, .
 $$
 
 Inserting this in the above (and using once again the basic estimate: $\Re \Delta (w, z)\lesssim\Re \Delta (w, z_\epsilon)$) we obtain
 
 $$
 |\,\Ein(f) (z_\epsilon) -\Ein (f)(z) \, | \leq C\,\epsilon^\beta\!\!\!\!\!
\int\limits_{w\in\bndry\Dm}\!\!\!\frac{|u_1|^{m-2}}{\Re \Delta (w, z)^{1+\beta}}\, \, d\sigma(w)  .
 $$
 
 Invoking one more time the convexity of $\Dm$ we conclude that the above integral is further bounded by
 
  $$
 C\,\epsilon^\beta\!\!\!\!\!\int\limits_{w\in\bndry\Dm}\!\!\!\!\!\!\frac{|u_1|^{m-2}}{|w-z|^{2+2\beta}}\, d\sigma(w)\leq
 C\epsilon^\beta\!\!\!\!\!\iiint\limits_{u_1^2+u_2^2+v_2^2<1}\!
 \frac{|u_1-x_1|^{m-2}}{(u_1^2+u_2^2+v_2^2)^{1+\beta}}\ \, du_1\, du_2\, dv_2.
$$
 The desired conclusion now follows by
  Lemma \ref{L:estimates} with $\alpha = m-1$ and with any
$0<\beta<(m-1)/2$.  
  
  Finally, we claim that 
 \begin{equation}\label{E:Rem-conv}
 \Rin(f) (z_\epsilon) -\Rin (f)(z)\to 0\quad \text{as}\ \ \epsilon\to 0
 \end{equation}
 uniformly in $z\in\bndry\Dm$. To see this we begin as before, with
 \bigskip
 $$
 |\, \Rin(f) (z_\epsilon) -\Rin (f)(z)\,|\, 
 \leq C\!\!\!
 \int\limits_{w\in\bndry\Dm}\!\!\!\!\!|u_1|^{2m-4}\,|w-z|^2
 \left|\frac{1}{\Delta^2 (w, z_\epsilon)}-\frac{1}{\Delta^2 (w, z)}\right| d\sigma(w)\, .
 $$
 
 By the basic estimate for $\Delta (w, z)$ we see that
 
 $$
 \left|\frac{1}{\Delta^2 (w, z_\epsilon)}-\frac{1}{\Delta^2(w, z)}\right|\leq \
 \frac{c_1\epsilon\,
  |\Delta (w, z_\epsilon)|}{|\Delta^2 (w, z_\epsilon)\, \Delta^2 (w, z)|}\ +\ 
\frac{c_1\epsilon\,
 |\Delta (w, z)|}{|\Delta^2 (w, z_\epsilon)\, \Delta^2 (w, z)|}\ \leq\ 
$$

$$
\leq 
\frac{c_1\epsilon
  }{\Re \Delta (w, z_\epsilon) (\Re \Delta (w, z))^2}\ +\ 
\frac{
c_1\epsilon
 }{(\Re \Delta (w, z_\epsilon))^2\, \Re \Delta (w, z)}\ \leq
$$

$$
\leq 
\frac{2\,c_1\epsilon
  }{\Re \Delta (w, z_\epsilon) (\Re \Delta (w, z))^2}\ =\ 
 \frac{2\,c_1\,\epsilon^\beta\, \epsilon^{1-\beta}
 }{\Re \Delta (w, z_\epsilon) (\Re \Delta (w, z))^2}\, .
$$
\smallskip
 
 \noindent Using again the trick: $\epsilon^{1-\beta}\leq (\Re \Delta (w, z_\epsilon))^{1-\beta}$
for any  $0\leq \beta<1$, we bound the latter with

$$
 \frac{2\,c_1\,\epsilon^\beta
 }{(\Re \Delta (w, z_\epsilon))^\beta (\Re \Delta (w, z))^2}\, \leq\ 
  \frac{2\,c_1\,\epsilon^\beta
  }{ (\Re \Delta (w, z))^{2+\beta}}.
$$
\smallskip

Finally, by the convexity of $\Dm$ we conclude that
 $$
 \left|\frac{1}{\Delta^2 (w, z_\epsilon)}-\frac{1}{\Delta^2(w, z)}\right|\leq \
 C\,
 \frac{\epsilon^\beta
 }{ |w-z|^{4+2\beta}}\, .
 $$
 
 Combining all of the above we conclude that
 
 $$
 |\, \Rin(f) (z_\epsilon) -\Rin (f)(z)\,|\, 
 \leq C\, \epsilon^\beta\!\!\!\!\!
 \int\limits_{w\in\bndry\Dm}\!\!\!\!\!\!\,
  \frac{|u_1|^{2m-4}
  }{ |w-z|^{2+2\beta}}\, d\sigma (w)\ \leq\ 
 $$
 
 $$
 \leq C\epsilon^\beta\!\!\!\!\!\iiint\limits_{u_1^2+u_2^2+v_2^2<1}\!\!\!\!\!\!\,
 \frac{|u_1-x_1|^{2m-4}}{(u_1^2+u_2^2+v_2^2)^{1+\beta}}\ du_1\, du_2\, dv_2\, .
 $$
 
  The desired conclusion now follows by
  applying 
   Lemma \ref{L:estimates} with $\alpha = 2m-3$ (which is in $(0, 1)$ thanks to our hypothesis that $3/2<m<2$) and with any
$0<\beta<\alpha/2=m-3/2$. 

This concludes the proof of the Proposition (assuming the truth of Lemma \ref{L:estimates}, whose proof is given below).
\end{proof}

 \begin{Lem}\label{L:estimates} Suppose that $|x_1|\leq 1$. Then we have that
 
 $$\displaystyle{I_{\alpha, \beta}:=\iiint\limits_{u_1^2+u_2^2+v_2^2<1}\!
 \frac{|u_1-x_1|^{-1+\alpha}}{\big(u_1^2+u_2^2+v_2^2\big)^{1+\beta}}\ \, du_1\, du_2\, dv_2}\ \leq A_{\alpha, \beta}<\infty$$
 \medskip
 
  \noindent is true for any $\alpha\in (0, 1]$ and any $0\leq\beta<\alpha/2$.
 The  constant $A_{\alpha, \beta}$ is independent of $x_1$.

 \end{Lem}
 \begin{proof}
Note that 
$I_{\alpha, 0}\leq I_{\alpha, \beta}$ for any $\beta>0$, thus we only need prove the conclusion 
in the case when $0<\beta<\alpha/2$.

 To begin with, we write
 $$
 \iiint\limits_{u_1^2+u_2^2+v_2^2<1}\!
 \frac{|u_1-x_1|^{-1+\alpha}}{\big(u_1^2+u_2^2+v_2^2\big)^{1+\beta}}\ \, du_1\, du_2\, dv_2\leq
 $$
 $$
= \int\limits_{|u_1|\leq 2}|u_1-x_1|^{-1+\alpha}
 \left(\ \, \iint\limits_{\mathbb R^2}
 \frac{1}{\big(u_1^2+u_2^2+v_2^2\big)^{1+\beta}}
 du_2\, dv_2\right)\, du_1\, .
 $$
 
 Notice that the change of variables: $u_2:= |u_1|\,\tilde u_2$
  and $v_2:= |u_1|\,\tilde v_2$, gives
 
  $$
  \iint\limits_{\mathbb R^2} \!\! \frac{1}{\big(u_1^2+u_2^2+v_2^2\big)^{1+\beta}}\, du_2\, dv_2  \ =\ 
 |u_1|^{-2\beta}\!\!\iint\limits_{\mathbb R^2} \!\! 
 \frac{1}{\big(1+\tilde u_2^2+\tilde v_2^2\big)^{1+\beta}}\, d\tilde u_2\, d\tilde v_2\leq c\,
  |u_1|^{-2\beta}
  $$
  
  with $c=c(\beta)$ (in particular $C$ is independent of $u_1$).
  Thus
  $$
  I_{\alpha, \beta}\, \leq c\!\!\! \int\limits_{|u_1|\leq\, 2}\!\!\! |u_1|^{-2\beta}|u_1-x_1|^{-1+\alpha}\, du_1\ =\  I\ +\ II\ +\  III\, ,
  $$
  where we have set
  $$
  I:=\!\!\!\!\!\!  \int\limits_{|u_1|\leq |x_1|/2}\!\!\!\!\!\!\!\!\! |u_1|^{-2\beta}|u_1-x_1|^{-1+\alpha}\, du_1\, ;\quad
   II :=\!\!\!\!\!\!\!\!\!\!\!\!  \int\limits_{|x_1|/2<|u_1|\leq 2\,|x_1|}\!\!\!\!\!\!\!\!\!\!\!\!\!\!\! |u_1|^{-2\beta}|u_1-x_1|^{-1+\alpha}\, du_1\, ,
  $$
  
  and
  
  $$
 III :=\!\!\!\!\!\!  \int\limits_{|u_1|>2\, |x_1|}\!\!\!\!\! |u_1|^{-2\beta}|u_1-x_1|^{-1+\alpha}\, du_1\, .
  $$
  \smallskip
  
  Now the active integrand factor for $A$ is $|u_1|^{-2\beta}$, that is
  $$
  I\leq |x_1|^{-1+\alpha}\!\!\!\!\!\int\limits_{|u_1|\leq |x_1|/2}\!\!\!\!\!\!\!\!\! |u_1|^{-2\beta}\, du_1
 = a |x_1|^{-1+\alpha} |x_1|^{1-2\beta} = a\, |x_1|^{\alpha -2\beta} = \ O(1)$$
 
 since $\alpha>2\beta$ (and $a=a(\beta)$, that is $a$ is independent of $x_1$).
\medskip
 
 On the other hand, the active integrand factor in $II$ is $|u_1-x_1|^{-1+\beta}$, that is
 
 $$
 II\leq \ 
 |x_1|^{-2\beta}\!\!\!\!\!\!\!\!\!\!\!\!  \int\limits_{|x_1|/2<|u_1|\leq 2\,|x_1|}\!\!\!\!\!\!\!\!\!\!\!\!\!\!\!
 |u_1-x_1|^{-1+\alpha}\, du_1
 \leq 
 |x_1|^{-2\beta}\!\!\!\!\!\!\!\int\limits_{|u_1|\leq 3|x_1|}\!\!\!\! |u_1|^{-1+\alpha}\,du_1 =
 b\, |x_1|^{-2\beta +\alpha} = \ O(1)\, 
 $$
 (again, here $b=b(\alpha)$ is independent of $x_1$).
 \medskip
 
 Finally
 
 $$
 III\ \leq \!\!\!\!\!\!  \int\limits_{|u_1|>2\, |x_1|}\!\!\!\!\! |u_1|^{-2\beta}|u_1|^{-1+\alpha}\, du_1
 \leq \int\limits_0^\infty |u_1|^{-1+\alpha-2\beta}du_1=c <\infty
 $$
 again because $\alpha-2\beta\in (0, 1)$ (and with $c=c (\alpha, \beta)$ independent of $x_1$).
 \end{proof}
 
While we are unable to prove that the Cauchy-Leray integral for the domain \eqref{E:star3} extends to a continuous function on the entire closure  $\bar \Dm$ when $1<m\leq 3/2$, we have the following result, valid for each $1<m<2$, whose proof will appear elsewhere.
\begin{Prop}\label{P:Dm-case2}
 Suppose that $f\in C^1(\bndry\Dm)$. Then $\Cin (f)$ extends to a continuous function in $\bar\Dm\setminus\big\{\bndry\Dm\, \cap\, \{x_1=0\}\big\}$. More precisely, for 
$z\in\bndry\Dm$ set
$$z_\epsilon = z +\epsilon\, \N (z)\ \in \Dm\, .$$ 

Then, we have that  $\Cin (f) (z_\epsilon)$ converges uniformly in
 $z\in \bndry\Dm \setminus \{ x_1=0\}$ to a limit that we denote $\Ctr (f)(z)$ which  satisfies
$$
|\Ctr (f) (z)|\leq C|x_1|^{-2+m}\, .
$$
\end{Prop}

\end{document}